\documentclass[12pt]{amsart}

\usepackage{amscd,amsthm}
\usepackage{latexsym}
\usepackage{amsmath,amstext,amsthm,amsfonts,amssymb}
\usepackage{graphicx}
\usepackage{newlfont}
\usepackage[ansinew]{inputenc}
\usepackage{graphpap}
\usepackage{mathrsfs}
\usepackage{anysize}

\usepackage{tikz}
\usepackage[all,cmtip]{xy}
\usetikzlibrary{decorations.markings}

\theoremstyle{plain}

\newtheorem{theorem}{Theorem}[section]
\newtheorem{proposition}[theorem]{Proposition}

\usepackage{xcolor}

\usepackage[active]{srcltx}%%%% Inverse search
\usepackage[colorlinks=true,linkcolor=red,urlcolor=black]{hyperref}

\theoremstyle{plain} %% This is the default

\usepackage[mathscr]{euscript}

\newcommand{\Rmc}{\mathbb{R}^+_0}
\newcommand{\R}{\mathbb{R}}

\usepackage[none]{hyphenat}

\author{Sebasti\'an Herrero}
\address{Instituto de Matem\'aticas, Pontificia Universidad Cat\'olica de Valpara\'iso, Blanco Viel 596, Cerro Bar\'on, Valpara\'iso,
Chile.}
\email{sebastian.herrero.m@gmail.com}

\author{Nelda Jaque}
\address{Departamento de Matem\'aticas, 
Facultad de Ciencias,
Universidad de Chile, 
Las Palmeras 3425, \~Nu\~noa, Santiago, Chile.}
\email{njaquetamblay@gmail.com}

\title{On two notions of expansiveness for continuous semiflows}

\begin{document}

\maketitle

\begin{abstract}
We study two notions of expansiveness for continuous semiflows: expansiveness in the sense of Alves, Carvalho and Siqueira (2017), and an adaptation of positive expansiveness in the sense of Artigue (2014). We prove that if $X$ is a metric space and $\phi$ is an expansive semiflow on $X$ according to the first definition, then the semiflow $\phi$ is trivial and the space $X$ is uniformly discrete. In particular, if $X$ is compact then it is finite. With respect to the second definition, we prove that if $X$ is a compact metric space and $\phi$ is a positive expansive semiflow on it, then $X$ is a union of at most finitely many closed orbits,  \emph{unbranched tails} and isolated singularities.
\end{abstract}

\section{Introduction}\label{sect-1}

Let $(X,d)$  denote a metric space. A continuous flow on $X$ is a continuous function $\Phi: X\times \mathbb{R}\to X$ %a continuous semiflow. This means that $\phi$ is a continuous function 
satisfying the following two conditions:
\begin{enumerate}
    \item[($i$)] $\Phi(x,t+s)=\Phi(\Phi(x,t),s)$ for all $t,s$ in $\R$ and all $x$ in $X$; 
    \item[($ii$)] $\Phi(x,0)=x$ for all $x\in X$.
\end{enumerate}
Analogously, a continuous semiflow on $X$ is a continuous function $\phi:~X\times\Rmc\to X$ satisfying  properties corresponding to $(i)$ and $(ii)$ above with $\R$ replaced by $\Rmc=[0,\infty[$. If $\Phi$ (resp.~$\phi$) is a flow (resp. semiflow) on $X$ then we use the notation $\Phi_t(x)=\Phi(x,t)$ (resp. $\phi_t(x)=\phi(x,t)$) for every~$t$ in~$\R$ (resp.~$\Rmc$) and~$x$ in~$X$. We say that~$\Phi$ (resp.~$\phi$) is \emph{trivial} if for every~$t$ in~$\R$ (resp.~$\Rmc$) we have that~$\Phi_t$ (resp.~$\phi_t$) is the identity map.

Clearly, every continuous flow $\Phi$ on $X$ defines, by restricting $t$ to $\Rmc$,  a continuous semiflow $\phi$ on the same space. %If $X$ supports non trivial flows, then 
On the other hand, it is easy to construct spaces admitting continuous semiflows that are not % that the procedure of restriction of flows does not give all possible semiflows on $X$. Namely, there exist continuous semiflows $\phi$ that are not 
the restriction to $\Rmc$ of any continuous flow.  This motivates the study of continuous semiflows and its comparison with the theory of continuous flows. The theory of continuous semiflows is also strongly motivated by the qualitative theory of differential equations (see, e.g., \cite{Sap81}).

In this paper we focus exclusively on the concept of expansiveness. First, we recall Bowen and Walters \cite{BW72} definition of  expansive flow.  Consider
$$C:=\{s:\mathbb{R}\to \mathbb{R}: s\text{ is continuous and }s(0)=0\}.$$
A continuous flow $\Phi:X\times \R\to X$ is called \begin{it}expansive\end{it} if for every $\varepsilon>0$ there exists $\delta>0$ such that for every pair of points $x,y$ in $X$ and every function $s$ in $C$ such that $d(\phi_t(x),\phi_{s(t)}(y))<\delta$ for all $t$ in $\mathbb{R}$, there exists  $t_0$ in $]-\varepsilon,\varepsilon[$ such that $y=\phi_{t_0}(x)$. This definition of expansiveness coincides with the notion of $\mathscr{F}$-expansiveness introduced by Keynes and Sears \cite{KS79} when $\mathscr{F}=C$. 

Inspired by Bowen and Walters' definition of expansiveness, Alves, Carvalho and Siqueira give in \cite[Section 1.2]{ACS17} the following definition. A continuous semiflow $\phi: X\times\Rmc\to X$ is called \begin{it}expansive\end{it} if it satisfies the following property:  For every $\varepsilon>0$ there exists $\delta>0$ such that for every pair of points $x,y$ in $X$ and every~$s$ in 
$$C^+:=\{s:\Rmc\to \Rmc:s\mbox{ is continuous and } s(0)=0\}$$
such that $d(\phi_t(x),\phi_{s(t)}(y))<\delta$ for all $t$ in $\Rmc$, there exists $t_0$ in $[0,\varepsilon[$ such that $y=\phi_{t_0}(x)$. This property is extremely  restrictive. Indeed, the first main result of this paper shows that, even when $C^+$ is replaced by 
$$H^+:=\{s\in C^+: s \text{ is surjective and strictly increasing}\}$$
(giving\begin{it} a priori\end{it} a more general definition), the only semiflow on $X$ satisfying such condition is the trivial one, and the expansiveness of this semiflow imposes constraints on the topology of $X$. 

\begin{theorem}\label{Theorem_1}
Let $(X,d)$ be a metric space. Assume that there exists a continuous semiflow $\phi:X\times \Rmc \to X$ satisfying the following property: For every $\varepsilon>0$ there exists $\delta>0$ such that for every pair of points $x,y$ in $X$ and every~$s$ in $H^+$ such that $d(\phi_t(x),\phi_{s(t)}(y))<\delta$ for all $t$ in $\Rmc$, there exists $t_0$ in $[0,\varepsilon[$ such that $y=\phi_{t_0}(x)$. Then, the following properties hold:  
\begin{enumerate}
\item[$(i)$] For every $t$ in $\Rmc$ we have that $\phi_t$ is the identity map on $X$.
\item[$(ii)$] The space $X$ is uniformly discrete. Namely, there exists $\rho>0$ such that $d(x,y)\geq \rho$ for every pair of points $x,y$ in $X$ with $x\neq y$. 
\end{enumerate}
In particular, if $X$ is compact then it is a finite set.
\end{theorem}

The second notion of expansiveness that we study in this paper is a symmetric version of the one given above. It is based on a definition given by Artigue in \cite[Section 2]{Art14} in the context of continuous flows without \emph{singularities}. Our definition is as follows: A continuous semiflow $\phi: X\times\Rmc\to X$ is called \begin{it}positive expansive\end{it} if for every $\varepsilon>0$ there exists $\delta>0$ such that for every pair of points $x,y$ in $X$ and every~$s$ in $H^+$ such that $d(\phi_t(x),\phi_{s(t)}(y))<\delta$ for all $t$ in $\Rmc$, there exists $t_0$ in $[0,\varepsilon[$ such that $y=\phi_{t_0}(x)$ or $x=\phi_{t_0}(y)$. 

In order to present our second main result, we introduce some definitions and notations. Given a space $X$, a point $x$ in $X$ and a semiflow $\phi: X\times\Rmc\to X$, the \emph{positive orbit} of $x$ under $\phi$ is
$$O_{\phi}^+(x)=\{\phi_t(x):t\in \Rmc\},$$
and the \emph{pre-orbit} of $x$ under $\phi$ is
$$O_{\phi}^-(x)=\{y\in X: \phi_t(y)=x \text{ for some }t \text{ in }\Rmc\}.$$
Define $\tau_x:O_{\phi}^-(x)\to \Rmc$ as the function
$$\tau_x(y):=\inf\{t\in \Rmc: \phi_t(y)=x\}.$$
The pre-orbit $O_{\phi}^-(x)$ is called \emph{unbranched} if the function $\tau_x$ is injective. \\
A point $x$ in $X$ is \emph{periodic} with respect to $\phi$ if there exists $t_0>0$ such that $\phi_{t_0}(x)=x$ and $O_{\phi}^+(x)\neq \{x\}$. In this case, the set $O_{\phi}^+(x)$ is a \emph{closed orbit} of $\phi$ and the \text{period} of $x$  is defined as
$$P:=\inf\{t\in\ ]0,\infty[:\phi_t(x)=x\}.$$
We also say that  $P$ is the period of $O_{\phi}^+(x)$. \\
A point $x$ in $X$ is called a \emph{singularity} (resp.~an \emph{end point}) of $\phi$ if $O_{\phi}^+(x)=O_{\phi}^-(x)=\{x\}$ (resp.~$O_{\phi}^+(x)=\{x\}$ and $O_{\phi}^-(x)\neq \{x\}$). If $x$ is an end point of $\phi$ then we call $O_{\phi}^-(x)$ the \emph{tail} of $x$. \\
The following examples illustrate the meaning of these concepts.

\bigskip

\noindent \begin{bf}Example 1. \end{bf}Consider the space $[0,1]$ with the usual euclidean metric, and the semiflow $\phi:[0,1]\times \Rmc\to [0,1]$ given  by $\phi_t(x)=\min\{x+t,1\}$ for $t$ in $\Rmc$ and $x$ in $[0,1]$. Then $\phi$ is a positive expansive semiflow with no singularities, one end point and no closed orbits. In this case the space consists of a single unbranched tail of $\phi$. \\

\noindent \begin{bf}Example 2. \end{bf}Consider again the space $[0,1]$ with the usual euclidean metric, and consider the semiflow $\varphi:[0,1]\times \Rmc\to [0,1]$ given  by $\varphi_t(x)=\min\{xe^t,1\}$ for $t$ in $\Rmc$ and $x$ in $[0,1]$. Then $0$ is a singularity and $O^-_{\varphi}(1)=\ ]0,1]$ is an unbranched tail of $\varphi$. By Theorem \ref{Theorem_2} below, we have that singularities of positive expansive semiflows are isolated. It follows that $\varphi$ is not  positive expansive. This is essentially due to the fact that points near $0$ move too slowly under the semiflow~$\varphi$. \\ 

\noindent \begin{bf}Example 3. \end{bf}Consider the cylinder
$$Y:=\left\{e^{2\pi i\theta}:\theta\in \mathbb{R}\right\}\times [0,1],$$
with the distance
$$d\left(\left(e^{2\pi i \theta_1},h_1\right),\left(e^{2\pi i \theta_2},h_2\right)\right):=|h_1-h_2|+\min\{|\theta_1-\theta_2+m|:m\in \mathbb{Z}\}.$$
Define $X$ as the subspace 
$$X:=\left\{\left(e^{2\pi i\theta},0\right):\theta\in \mathbb{R}\right\} \cup \left\{\left(e^{2\pi i \theta},e^{\theta}\right): \theta\leq 0 \right\}.$$
Consider the semiflow $\psi: X\times \Rmc \to X$ given for $t$ in $\Rmc$ by
$$\psi_t\left(e^{2\pi i \theta},h \right)=\left\{
\begin{array}{ll}
\left(e^{2\pi i (\theta+t)},he^t\right) & \text{ if }he^t\leq 1,\\
(1,1) & \text{ if }he^t> 1.
\end{array}
\right.$$
Then $\psi$ is a continuous semiflow on $X$. We claim that $\psi$ is positive expansive. Indeed, given $\varepsilon>0$ choose $\delta=\min \left\{\varepsilon,\frac{1}{4} \right\}$ and let $x_1=(e^{2\pi i \theta_1},h_1), x_1=(e^{2\pi i \theta_2},h_2)$ in $X$ and $s$ in $H^+$ be such that $d(\psi_t(x_1),\psi_{s(t)}(x_2))<\delta$ all $t$ in $\Rmc$. For every such $t$, let $m_t$ be an integer such that
$$|\theta_1+t-(\theta_2+s(t))+m_t|=\min\{|\theta_1+t-(\theta_2+s(t))+m|:m\in \mathbb{Z}\}.$$
Then
\begin{equation}\label{ec-t-s(t)}
|\theta_1+t-(\theta_2+s(t))+m_t|\leq d(\psi_t(x_1),\psi_{s(t)}(x_2))<\delta,
\end{equation}
hence
$$m_t\in \ ]\theta_2-\theta_1+s(t)-t-\delta,\theta_2-\theta_1+s(t)-t+\delta[$$
for all $t$. Since the above interval has length $2\delta<\frac{1}{2}$ and $s(t)-t$ varies continuously on $t$, it follows that $m_t$ is constant equal to $m_0$. Now, if $h_1=0$ then clearly $h_2=0$ and in this case we have $\psi_{\theta_1-\theta_2+m_0}(x_2)=x_1$ if $\theta_1-\theta_2+m_0\geq 0$, and $\psi_{\theta_2-\theta_1-m_0}(x_1)=x_2$ if $\theta_1-\theta_2+m_0<0$. Since $|\theta_1-\theta_2+m_0|<\delta\leq \varepsilon$, the desired property holds. Now, assume $h_1> 0$. Then $h_2>0$ and without loss of generality we can assume $h_i=e^{\theta_i}$ for $i=1,2$. Define
$$T:=\min \left\{|\theta_1|,s^{-1}(|\theta_2|) \right\}.$$
Note that $T$ is the first time $t$ when $\psi_t(x_1)$ or $\psi_{s(t)}(x_2)$ reaches $(1,1)$. If $T=|\theta_1|$, then $\psi_T(x_1)=(1,1)$ and
$$d((1,1),\psi_{s(T)}(x_2))=\left(1-e^{\theta_2+s(T)}\right)+|\theta_2+s(T)-m_0|<\delta.$$
In particular, $1-e^{\theta_2+s(T)}<\delta\leq \frac{1}{4}$ and $|\theta_2+s(T)-m_0|\leq \delta <\frac{1}{4}$. Since
$$-\frac{1}{2}<\log\left(\frac{3}{4}\right)\leq \theta_2+s(T)\leq 0,$$
we have $m_0=0$, thus by \eqref{ec-t-s(t)} with $t=0$ we get $|\theta_1-\theta_2|<\delta\leq \varepsilon$. Since $\psi_{\theta_1-\theta_2}(x_2)=x_1$ if $\theta_1-\theta_2\geq 0$, and $\psi_{\theta_2-\theta_1}(x_1)=x_2$ if $\theta_1-\theta_2< 0$, we get the desired result. The case $T=s^{-1}(|\theta_2|)$ is similar. This proves that $\psi$ is positive expansive.\\
Note that in this example $\psi$ has no singularities and
$$X=O_{\psi}^+((1,0)) \cup O_{\psi}^-((1,1)),$$
where $O_{\psi}^+((1,0))$ is a closed orbit of period $1$ and $O_{\psi}^-((1,1))$ is an unbranched tail. Note also that the tail $O_{\psi}^-((1,1))$ is open and dense in $X$.

\bigskip

We now present the second main result of this paper.

\begin{theorem}\label{Theorem_2}
Let $(X,d)$ be a compact metric space. Assume that $\phi:X\times \Rmc \to X$ is a continuous semiflow that is positive expansive. Then, $X$ is a union of at most finitely many closed orbits,  unbranched tails and isolated singularities.
\end{theorem}

Theorems \ref{Theorem_1} and \ref{Theorem_2} can be seen as part of a collection of results showing that certain expansive dynamical systems are trivial. For instance, a theorem first proved by Schwartzman \cite{Sch53} states that \begin{it}positively expansive homeomorphisms\end{it} on compact metric spaces are trivial; see also \cite{RW04,CK06,MS10} for definitions and further references. In the context of continuous dynamical systems, Artigue \cite{Art14} has shown that \begin{it}positive expansive flows\end{it} on compact metric spaces consist of at most finitely many singularities and closed orbits. We also cite \cite{MS19} for related finiteness results for certain \begin{it}expansive group actions\end{it}. 

\subsection{Organization} The proof of Theorem \ref{Theorem_1} is given in section \ref{sec-th-1} and it is based on elementary properties of continuous functions on metric spaces. In section \ref{sec-general-expansive} we prove  some general properties of positive expansive semiflows on not necessarily compact metric spaces.  The proof of Theorem \ref{Theorem_2}  is based on these general properties, together with Artigue's description of positive expansive flows without singularities on compact metric spaces    (\cite[Theorem 4.2]{Art14}).  This is  presented in section \ref{sec-th-2}.

\subsection{Comments} During the preparation of this paper the authors learned about the related work of Lee, Lee and Morales \cite{LLM20}. One of their results states that if $\phi$ is an expansive semiflow on a compact metric space $X$, then $X$ is a union of at most finitely many singularities and  closed orbits. Such result can be seen as a weak form of Theorem \ref{Theorem_1}. Lee, Lee and Morales also study an alternative definition of expansiveness (called \emph{eventual expansiveness}) that allows the existence of branched tails, and prove an analogue of Theorem \ref{Theorem_2} for such semiflows under certain additional assumptions.

\section{Proof of Theorem \ref{Theorem_1}}\label{sec-th-1}

Let $(X,d)$ be a metric space and let $\phi$ be a continuous semiflow on $X$ satisfying the hypothesis of Theorem \ref{Theorem_1}. By contradiction, assume that there exists $T$ in $\Rmc$ such that $\phi_T$ is not the identity map on $X$. Since $\phi_0$ is the identity map on $X$, we have that $T$ is strictly positive. By definition of $T$ there exists a point $z$ in $X$ such that $\phi_{T}(z)\neq z$. Define $T_0:=\tau_{\phi_{T}(z)}(z)$. %\inf\{t\in \Rmc: \phi_t(z)=\phi_T(z)\}.$$
Then $T_0>0$ and $\phi_{T_0}(z)=\phi_T(z)\neq z$.  Since $\phi$ is a semiflow, for every $s$ in $[0,T_0]$ we have that $\tau_{\phi_{T}(z)}(\phi_s(z))=T_0-s$. 
%$$\inf\{t\in \Rmc: \phi_t(\phi_s(z))=\phi_{T_0}(z)\}=T_0-s.$$
This implies that the function $\phi(\ast,z):[0,T_0]\to X$ given by $t\mapsto  \phi(t,z)$ is injective. Choose $\varepsilon:=\frac{T_0}{2}$ and let $\delta$ be the corresponding constant given in the statement of the theorem. Define $T_1:=\frac{T_0}{4}$ and $x:=\phi_{T_1}(z)$. We observe that %$d(z,x)>0$ and 
the function
$$I:[0,T_0] \to \Rmc, I(t):=d(\phi_t(z),x)$$
is continuous and  $I(T_1)=0$. 
It follows from the continuity of $I$ that there exists $\delta_1$ in $]0,T_1[$ such that $I(t)<\frac{\delta}{2}$ for all $t$ in $[T_1-\delta_1,T_1+\delta_1]$. Put $y:=\phi_{T_1-\delta_1}(z)$ and define $s:\Rmc\to\Rmc$ by
$$s(t)=\left\{\begin{array}{ll}
%\left(\frac{T_1+\delta_1}{\delta_1}\right)
2t & \mbox{ if }t\leq \delta_1,\\
t+\delta_1 & \mbox{ if }t> \delta_1.
\end{array}
\right.$$  
Note that $s$ is in $H^+$. Now, for every $t$ in $[0,\delta_1]$ we have that $T_1+t$ and $s(t)+T_1-\delta_1$ are both in $[T_1-\delta_1,T_1+\delta_1]$, hence 
$$d(\phi_{t}(x),x)=d(\phi_{t}(\phi_{T_1}(z)),x)=I(T_1+t)<\frac{\delta}{2}$$
and
$$d(\phi_{s(t)}(y),x)=d(\phi_{s(t)}(\phi_{T_1-\delta_1}(z)),x)=I(s(t)+T_1-\delta_1)<\frac{\delta}{2}.$$ 
This implies, by the triangular inequality, that $d(\phi_t(x),\phi_{s(t)}(y))<\delta$ for all $t$ in $[0,\delta_1]$. On the other hand, for $t$ in $]\delta_1,\infty[$ we have 
$$\phi_t(x)=\phi_{t+T_1}(z)=\phi_{t+\delta_1}(\phi_{T_1-\delta_1}(z))=\phi_{s(t)}(y),$$
hence $d(\phi_t(x),\phi_{s(t)}(y))=0$. We conclude that $d(\phi_t(x),\phi_{s(t)}(y))<\delta$ for all $t$ in $\Rmc$. By hypothesis, we get that $y=\phi_{t_0}(x)$ for some $t_0$ in $[0,\varepsilon[$, hence
$$\phi_{T_1-\delta_1}(z)=y=\phi_{t_0}(x)=\phi_{t_0+T_1}(z),$$
with $T_1-\delta_1$ and $t_0+T_1$ both in $[0,T_0]$. This contradicts the fact that $\phi(\ast,z):[0,T_0]\to X$ is injective since $T_1-\delta_1<t_0+T_1$. This proves $(i)$.\\
In order to prove $(ii)$, choose $\varepsilon':=1$ and let $\delta'$ be the corresponding constant given in the statement of the theorem. If $(ii)$ does not hold, then there exist two distinct points $x,y$ in $X$ with $d(x,y)<\delta'$.  Choosing any $s$ in $H^+$ we have, by $(i)$, that $d(\phi_t(x),\phi_{s(t)}(y))=d(x,y)<\delta'$ for all $t$ in $\Rmc$, hence there exists $t_0$ in $[0,1[$ with $y=\phi_{t_0}(x)$. This implies that $y=x$, giving a contradiction. This proves $(ii)$.\\
The last statement of the theorem follows directly from $(ii)$, since every compact and discrete space is finite. This completes the proof of Theorem \ref{Theorem_1}.

\section{General properties of positive expansive semiflows}\label{sec-general-expansive}

In this section we prove some general properties of positive expansive semiflows on not necessarily compact metric spaces. These properties are presented in Proposition \ref{general-positive-expansive} bellow. 

Given a space $X$, a semiflow $\phi$ on $X$ and a subset $Y$ of $X$, we say that $Y$ is \emph{positive invariant} under $\phi$ if $\phi_t(Y)\subseteq Y$ for all $t$ in $\Rmc$. Moreover, we define $X_{\mathrm{sing}}$ as the set of singularities and  $X_{\mathrm{tail}}$ as the union of all tails of $\phi$ in $X$.

\begin{proposition}\label{general-positive-expansive}
Let $(X,d)$ be a metric space and let $\phi: X \times \Rmc \to X$ be a continuous semiflow that is positive expansive. Then:
\begin{enumerate}
\item[$(i)$]  For every $x$ in $X$, the pre-orbit $O^-_{\phi}(x)$ is unbranched. In particular, every tail of $\phi$ is unbranched. 
\item[$(ii)$] Every singularity of $\phi$  is isolated in $X$. In particular, $X_{\mathrm{sing}}$ is open in $X$.
\item[$(iii)$] Every tail of $\phi$  is open in $X$. In particular, $X_{\mathrm{tail}}$ is open in $X$.
\item[$(iv)$] $X\setminus (X_{\mathrm{sing}}\cup X_{\mathrm{tail}})$ is positive invariant under $\phi$.
\item[$(v)$] For every $x$ in $X$, if there exists $t$ in $\Rmc$ such that $\phi_t(x)$ is periodic, then $x$ is periodic.
\item[$(vi)$] For every $t$ in $\Rmc$, the map $\phi_t:X\setminus X_{\mathrm{tail}} \to X\setminus X_{\mathrm{tail}}$ is injective.
\end{enumerate} 
\end{proposition}

In the following proof, we use $B(x,\delta)$ to denote the open ball in $X$ of center $x$ an radius $\delta$.

\begin{proof}[Proof of Proposition \ref{general-positive-expansive}]
We start with item $(i)$. Let $x$ be a point in $X$ and let $z_1,z_2$ be points in $O^-_{\phi}(x)$ such that $\tau_x(z_1)=\tau_x(z_2)$. Define
$$T:=\inf\{t\in \Rmc:\phi_t(z_1)=\phi_t(z_2)\}.$$
%$$T_1:=\inf\{t\in \Rmc:\phi_t(z_1)\in O^+_{\phi}(z_2)\text{ or }\phi_t(z_2)\in O^+_{\phi}(z_1)\}.$$
Note that $T\leq \tau_x(z_1)$. We claim that $T=0$. Indeed, assume this is not the case. Choose $\varepsilon:=1$ and let $\delta_1>0$ be the corresponding constant given by positive expansiveness of $\phi$.  
Since $\phi$ is continuous, there exists $\mu$ in $]0,T[$ such that $d(\phi_{t_1}(z_1),\phi_{t_2}(z_2))<\delta_1$ for all $t_1$ in $[\mu,T]$ and $t_2$ in $[\mu,T]$. If we define $w_i:=\phi_{\mu}(z_i)$ for $i=1,2$, then we have
$$d(\phi_t(w_1),\phi_t(w_2))<\delta_1$$
for all $t$ in $[0,T-\mu]$, and $\phi_t(w_1)=\phi_t(w_2)$ for all $t\geq T-\mu$. By expansiveness, there exists $t_0$ in $[0,1[$ such that $\phi_{t_0}(w_1)=w_2$ or $\phi_{t_0}(w_2)=w_1$. This implies $\tau_x(w_1)=\tau_x(w_2)+t_0$ or $\tau_x(w_2)=\tau_x(w_1)+t_0$. Since 
$$\phi_{T-\mu}(w_1)=\phi_T(z_1)=\phi_T(z_2)=\phi_{T-\mu}(w_2),$$
we also have $\tau_x(w_1)=\tau_x(w_2)$. This implies that $t_0=0$ hence $w_1=w_2$. Thus $\phi_{\mu}(z_1)=\phi_{\mu}(z_2)$, contradicting the definition of $T$. This proves that $T=0$, which implies $z_1=z_2$. We conclude that $\tau_x$ is injective, as desired. This proves that the pre-orbit $O^-_{\phi}(x)$ is unbranched. Since every tail is a pre-orbit, we have that every tail is unbranched. This completes the proof of $(i)$.\\
We now prove $(ii)$ and $(iii)$ simultaneously. Assume that $x_0$ is a singularity or an end point of $\phi$ in $X$. As before, let $\delta_1>0$ be the constant given by positive expansiveness of $\phi$ corresponding to $\varepsilon=1$. Since the restriction of $\phi$ to $X\times [0,2]$ is uniformly continuous, there exists $\delta_2>0$ such that for every $x_1,x_2$ in $X$ and $t_1,t_2$ in $[0,2]$ we have
\begin{equation}\label{cont-unif}
d(x_1,x_2)<\delta_2, |t_1-t_2|<\delta_2 \Rightarrow d(\phi_{t_1}(x_1),\phi_{t_2}(x_2))<\frac{\delta_1}{2}.
\end{equation}
Put $\delta_3:=\min \left\{\frac{\delta_1}{2},\delta_2\right\}$ and 
$$V:=B(x_0,\delta_3)\cap \phi_1^{-1}(B(x_0,\delta_3)).$$
Note that $V$ is an open neighbourhood of $x_0$. We claim that $V$ is contained in $O^-_{\phi}(x_0)$. Indeed, let $y$ be a point in $V$ and put $x:=\phi_1(y)$. Define $s:\Rmc\to \Rmc$ by
$$s(t)=\left\{\begin{array}{ll}
2t & \text{ if }t\in [0,1],\\
t+1 & \text{ if }t>1.
\end{array}\right.$$
On one hand, if $t$ is in $[0,1]$ then $s(t)$ is in $[0,2]$, hence by \eqref{cont-unif} we have
$$d(x_0,x)<\delta_3\leq \delta_2 \Rightarrow d(x_0,\phi_t(x))<\frac{\delta_1}{2}$$
and
$$d(x_0,y)<\delta_3\leq \delta_2 \Rightarrow d(x_0,\phi_{s(t)}(y))<\frac{\delta_1}{2}.$$
This implies that $d(\phi_t(x),\phi_{s(t)}(y))<\delta_1$ for all $t$ in $[0,2]$. On the other hand, for $t>1$ we have $\phi_{s(t)}(y)=\phi_{t+1}(y)=\phi_t(x)$. We conclude that $d(\phi_t(x),\phi_{s(t)}(y))<\delta_1$ for all $t$ in $\Rmc$. By positive expansiveness, we get that there exists $t_0$ in $[0,1[$ such that $y=\phi_{t_0}(x)$ or $x=\phi_{t_0}(y)$. In the first case we have that $y$ is a singularity or a periodic point of $\phi$ of period at most $1+t_0$. By \eqref{cont-unif} we have $d(\phi_t(y),x_0)<\delta_1$ for all $t$ in $\Rmc$, hence by positive expansiveness $y$ is in $O^-_{\phi}(x_0)$. This implies that $x$ is in $O^-_{\phi}(x_0)$ as desired. In the second case, we have that $x$ is a singularity or a periodic point of $\phi$ of period at most $1-t_0$. By \eqref{cont-unif} we get  $d(\phi_t(x),x_0)<\delta_1$ for all $t$ in $\Rmc$, and by positive expansiveness $x$ is in $O^-_{\phi}(x_0)$. This proves that $V$ is contained in $O^-_{\phi}(x_0)$. Since $x_0$ is in $V$, it follows that
$$\bigcup_{t \in \Rmc}\phi_t^{-1}(V)=O^-_{\phi}(x_0).$$
Since $V$ is open, this implies that $O^-_{\phi}(x_0)$ is open. If $x_0$ is a singularity, then $O^-_{\phi}(x_0)=\{x_0\}$ is open. Hence $x_0$ is isolated in $X$. Since $x_0$ is arbitrary, this implies that $X_{\mathrm{sing}}$ is open in $X$. If $x_0$ is an end point, we get that the tail $O^-_{\phi}(x_0)$ is open. Since $X_{\mathrm{tail}}$ is, by definition, a union of tails, we conclude that $X_{\mathrm{tail}}$ is open. This completes the proof of $(ii)$ and $(iii)$.\\
The proof of item $(iv)$ is straightforward. \\
In order to prove $(v)$, fix $t$ in $\Rmc$ and let $u$ be a point in $X$ such that $\phi_t(u)$ is periodic. Let $P>0$ be the period of $\phi_t(u)$, define
$$T_0:=\inf\{\alpha\in \Rmc:\phi_{\alpha}(u)\in O^+_{\phi}(\phi_t(u))\},$$
and put $w:=\phi_{T_0}(u)$. Then $O^+_{\phi}(w)=O^+_{\phi}(\phi_t(u))$. If $T_0>0$, choose $\beta:=\min\{P,T_0\}$. Then $\tau_w(\phi_{T_0-\beta}(u))=\beta=\tau_w(\phi_{P-\beta}(w))$. By $(i)$ we have $\phi_{T_0-\beta}(u)=\phi_{P-\beta}(w)\in O^+_{\phi}(w)=O^+_{\phi}(\phi_t(u))$. But this contradicts the definition of $T_0$. This proves that $T_0=0$, hence $u \in O^+_{\phi}(\phi_t(u))$. This implies that $u$ is periodic, as wanted. This proves $(v)$.\\
We now prove $(vi)$. Fix $t$ in $\Rmc$ and let $u_1,u_2$ be distinct points in $X\setminus X_{\mathrm{tail}}$  such that $\phi_t(u_1)=\phi_t(u_2)$. Define $r_i:=\tau_{\phi_t(u_1)}(u_i)$ for $i=1,2$. By $(i)$ we have $r_1\neq r_2$. Without loss of generality we can assume $r_2<r_1$. Then $\tau_{\phi_t(u_1)}(\phi_{r_1-r_2}(u_1))=r_2=\tau_{\phi_t(u_1)}(u_2)$. By $(i)$ we have $\phi_{r_1-r_2}(u_1)=u_2$, thus 
$$\phi_{r_1-r_2}(\phi_t(u_2))=\phi_{r_1-r_2}(\phi_t(u_1))=\phi_{t}(\phi_{r_1-r_2}(u_1))=\phi_t(u_2).$$ 
Since $u_2$ is not in $X_{\mathrm{tail}}$, this implies that $\phi_t(u_2)$ is periodic. By $(v)$ we have that $u_2$ is periodic. Since $\phi_{r_1-r_2}(u_1)=u_2$, by $(v)$ we have that $u_1$ is also periodic and $O_{\phi}^+(u_1)=O_{\phi}^+(u_2)$. Clearly, $\phi_t$ is injective in each closed orbit of $\phi$. In particular, $\phi_t$ is periodic in  $O_{\phi}^+(u_1)$, thus $u_1=u_2$. This proves $(vi)$ and completes the proof of the proposition.
\end{proof}

\noindent \begin{bf}Remark 1.\end{bf} Item $(vi)$ in Proposition \ref{general-positive-expansive} does not hold in general if we replace $X\setminus X_{\mathrm{tail}}$ by $X$. Indeed, Example 1 in section \ref{sect-1} gives an instance of a semiflow $\phi$ such that for every $t>0$, the map $\phi_t$ is not injective. \\

\noindent \begin{bf}Remark 2.\end{bf} It follows from Theorem \ref{Theorem_2} that if $(X,d)$ is a compact metric space and $\phi$ is a positive expansive semiflow on $X$, then for every $t$ in $\Rmc$ the map $\phi_t:X\setminus X_{\mathrm{tail}}\to  X\setminus X_{\mathrm{tail}}$ is surjective (since $X\setminus X_{\mathrm{tail}}$ consists of singularities and closed orbits). This does not hold in general if $X$ is not compact. For an example, consider the semiflow $\phi:\Rmc \times \Rmc\to \Rmc$ given by $\phi_t(x)=x+t$ for $x$ and $t$ in $\Rmc$. Then, for every $t>0$ we see that $\phi_t$ is not surjective.

\section{Proof of Theorem \ref{Theorem_2}}\label{sec-th-2}

Let $(X,d)$ be a compact metric space, and let $\phi:X\times \Rmc\to X$ be a positive expansive semiflow. Define $X_0$ as the set of singularities and  end points of $\phi$ in $X$. First, we prove that $X_0$ is finite. Assume this is not the case, and let $(x_n)_{n=1}^{\infty}$ be a sequence of distinct points in $X_0$.  Choose $\varepsilon:=1$ and let $\delta$ be the corresponding constant given by positive expansiveness of $\phi$. Since $X$ is compact, the sequence $(x_n)_{n=1}^{\infty}$ has a convergent subsequence. This implies that there exists distinct positive integers $n,m$ such that $d(x_n,x_{m})<\delta$. Since $x_n$ and $x_m$ are in $X_0$, we have $d(\phi_t(x_n),\phi_t(x_{m}))=d(x_n,x_m)<\delta$ for all $t$ in $\Rmc$. By positive expansiveness, there exists $t_0$ in $[0,1[$ such that $\phi_{t_0}(x_n)=x_m$ or  $\phi_{t_0}(x_m)=x_n$. In any case we get $x_n=x_m$, contradicting the fact that $x_n$ and $x_m$ are distinct. This proves that the set $X_0$ is finite. In particular, $X_{\mathrm{sing}}$ is finite and $X_{\mathrm{tail}}$ is a union of at most finitely many tails of $\phi$. Moreover, by Proposition \ref{general-positive-expansive}$(i)$ and $(ii)$, each tail is unbranched and  each singularity is isolated. \\
In order to complete the proof of Theorem \ref{Theorem_2}, it is enough to prove that $Y:=X\setminus (X_{\mathrm{sing}}\cup X_{\mathrm{tail}})$ is a union of at most finitely many closed orbits of $\phi$. If $Y$ is empty, then there is nothing to prove. Assume $Y$ is not empty. By Proposition \ref{general-positive-expansive}$(ii)$, $(iii)$ and $(iv)$, $Y$ is a closed positive invariant subset of $X$. In particular, $Y$ is compact.  Define
$$\widetilde{Y}:=\bigcap_{t\in \Rmc}\phi_t(Y).$$
Since  $\widetilde{Y}$ is the intersection of a family of compact sets with the finite intersection property, we have that $\widetilde{Y}$ is compact and not empty. Moreover, $\widetilde{Y}$ is positive invariant and for each $t$ in $\Rmc$, the map $\phi_t:\widetilde{Y}\to \widetilde{Y}$ is surjective. By Proposition \ref{general-positive-expansive}$(vi)$ and the fact that $\widetilde{Y}$ is compact, we conclude that for every $t$ in $\Rmc$, the map $\phi_t:\widetilde{Y}\to \widetilde{Y}$ is a homeomorphism. Hence, the map $\Phi:\widetilde{Y}\times \Rmc\to \widetilde{Y}$ defined for $y$ in $\widetilde{Y}$ and $t$ in $\mathbb{R}$ by
$$\Phi(y,t)=\left\{\begin{array}{ll}
\phi_t(y) \text{ if }t\geq 0,\\
\phi_{-t}^{-1}(y) \text{ if }t< 0,
\end{array}
\right.$$
is a continuous flow on $\widetilde{Y}$ (see, e.g., \cite[Lemma 2.4]{Sap81}). The flow $\Phi$ has no singularities and it is positive expansive in the sense of \cite[Definition~2.1]{Art14}. 
%Since $\phi$ is positive expansive, we have that $\Phi$ is a positive expansive flow in the sense of \cite[Section 2]{Art14}. 
By \cite[Theorem 4.2]{Art14} we get that $\widetilde{Y}$ is a union of finitely many closed orbits of $\Phi$. Thus,  $\widetilde{Y}$ is a union of closed orbits $C_1,\ldots, C_M$ of $\phi$.\\
We now prove that $Y=\widetilde{Y}$. Let $y$ be a point in $Y$, and consider the \emph{positive limite set}
$$L^+(y):=\{z\in Y: \phi_{t_n}(y)\to z \text{ for some sequence }(t_n)_{n=1}^{\infty} \text{ in }\Rmc \text{ with }t_n\to \infty\}.$$
Then, $L^+(y)$ is a closed, positively invariant, connected, not empty subset of~$Y$ (see, e.g., \cite[Theorems~3.4 and~3.5]{Sap81}). Note that for every $z$ in $L^+(y)$ and $t$ in $\Rmc$, there is a sequence $(t_n)_{n=1}^{\infty}$ in $[t,\infty[$ with $\phi_{t_n}(y)\to z$, and the  sequence $(\phi_{t_n-t}(y))_{n=1}^{\infty}$ has a subsequence converging to some point $z_t$ in $Y$ with $\phi_t(z_t)=z$. This proves that $L^+(y)$ is contained in $\widetilde{Y}$. Therefore, $L^+(y)=C_i$ for some $i$ in $\{1,\ldots,M\}$. Let $P$ be the period of $C_i$ and let $\delta'$ be the constant given by positive expansiveness of $\phi$ corresponding to $\varepsilon':=P$. We claim that there exists $T$ in $\Rmc$ such that $d(\phi_{t}(y),\phi_{P+t}(y))<\delta'$ for all $t$ in $[T,\infty[$. Indeed, if this were not the case, then there would exist a sequence $(T_n)_{n=1}^{\infty}$ in $\Rmc$ such that $T_n\to \infty$ and $d(\phi_{T_n}(y),\phi_{P+T_n}(y))\geq \delta'$. By compactness of $Y$ and continuity of $\phi_{P}$, this implies the  existence of a point $z_0$ in $L^+(y)=C_i$ with $d(z_0,\phi_{P}(z_0))\geq \delta'>0$. This is impossible since $z_0=\phi_{P}(z_0)$. Thus, there exists $T$ in $\Rmc$ such that $d(\phi_{t}(y),\phi_{P+t}(y))<\delta'$ for all $t$ in $[T,\infty[$. By expansiveness, this implies that there exists $\mu$ in $[0,P[$ such that $\phi_{\mu}(\phi_{T}(y))=\phi_{P+T}(y)$ or  $\phi_{\mu}(\phi_{P+T}(y))=\phi_{T}(y)$. Thus $\phi_{\mu+T}(y)$ or $\phi_{T}(y)$ is periodic. By Proposition \ref{general-positive-expansive}$(v)$ we get that $y$ is  periodic. This implies that $y$ is in $\widetilde{Y}$ as desired. This proves that $Y=\widetilde{Y}$ and completes the proof of Theorem \ref{Theorem_2}.

\section*{Acknowledgements} The authors would like to thank C.~A.~Morales for sharing with them a copy of the preprint \cite{LLM20}.  S.~Herrero was partially supported by Proyecto VRIEA-PUCV 039.344/2021. N.~Jaque was partially supported by ANID/CONICYT FONDECYT Iniciaci\'on grant 11190815.

\bibliographystyle{alpha}
\bibliography{references}
\end{document}